\documentclass[11pt]{amsart}

\usepackage{amsmath,amssymb,amsthm}

\usepackage{fancyhdr}
\usepackage{amsfonts,graphicx}

\theoremstyle{plain}
\newtheorem{theo}{Theorem}[section]
\newtheorem{lemm}[theo]{Lemma}

\newtheorem{prop}[theo]{Proposition}

\theoremstyle{plain}
\theoremstyle{definition}

\theoremstyle{remark}
\newtheorem{rema}[theo]{Remark}
\newtheorem*{rema*}{Remarks}

\newcommand{\ZZ}{\mathbb{Z}}  
  
\newcommand{\NN}{\mathbb{N}}

\newcommand{\RR}{\mathbb{R}}

\newcommand{\ww}{\tilde{\omega}}

\newcommand{\diver}[1]{ \textnormal{div}\hspace{0.07cm} #1} 

\numberwithin{equation}{section}
\author[H. Abidi]{Hammadi Abidi}
\address{IRMAR, Universit\'e de Rennes 1\\ Campus de
Beaulieu\\ 35~042 Rennes cedex\\ France}
\email{hamadi.abidi@univ-rennes1.fr}
\author[T. Hmidi]{Taoufik Hmidi}
\address{IRMAR, Universit\'e de Rennes 1\\ Campus de
Beaulieu\\ 35~042 Rennes cedex\\ France}
\email{thmidi@univ-rennes1.fr}
\author[S. Keraani]{Sahbi Keraani}
\address{IRMAR, Universit\'e de Rennes 1\\ Campus de
Beaulieu\\ 35~042 Rennes cedex\\ France}
\email{sahbi.keraani@univ-rennes1.fr}
\title[Global  well-posedness for the  axisymmetric Euler equation]
{On the global well-posedness for the axisymmetric Euler equations}
 \date{}
\begin{document}
\begin{abstract}
This paper deals with the global well-posedness of  the $3$D
axisymmetric Euler equations for  initial data lying in  critical Besov spaces 
$B_{p,1}^{1+3/p}$.   In this case the BKM criterion is not known to be valid and to circumvent this difficulty we use a new decomposition of the vorticity.   
\end{abstract}

\maketitle
\section{Introduction}
\hskip 12pt
The evolution of homogeneous inviscid incompressible fluid flows \mbox{in 
$\mathbb R^3$} is governed by the
Euler system
\begin{equation}
\label{E}
\left\lbrace
\begin{array}{l}
\partial_t u+(u\cdot\nabla)u+\nabla \pi=0,\\
\diver u=0,\\
{u}_{| t=0}=u_0.  
\end{array}
\right.  
\end{equation}
Here,  $u=u(t,x)\in \mathbb R^3$ denotes the velocity of the fluid, the scalar \mbox{function  
$\pi=\pi(t,x)$}  stands for  the  scalar pressure and $u\cdot\nabla=\sum_{j=1}^3u^j\partial_j$.  

The local  theory of the system \eqref{E} seems to be in a satisfactory state and several  results are obtained by numerous authors  in many standard function spaces.   In \cite{Kato}, Kato proved the local existence  and uniqueness for initial \mbox{data $u_0\in H^s(\mathbb R^3)$} with $s>5/2$ and \mbox{Chemin \cite{Ch1}} gave similar results for initial data lying in H\"olderian \mbox{spaces $C^r$} 
with $r>1.  $

Other local results are recently obtained by \mbox{Chae \cite{Chae}} in critical Besov \mbox{spaces 
$B_{p,1}^{1+3/p},$} with $p\in]1,\infty[$ and by Pak and Park \cite{Hee} for the \mbox{space $B_{\infty,1}^1$. }  
Notice that  these spaces   have the same scaling as  Lipschitz functions (the space which is relevant for the hyperbolic theory)  and in this sens they are called  critical.  

The question of global existence (even for a smooth initia data) is still
open
and continues to be  one of the most leading problem in mathematical fluid mechanics.    The well-known BKM criterion  \cite{Beale}  ensures that the development of finite time singularities for Kato's solutions is related to the blowup  of the $L^\infty$ norm of the vorticity near the maximal time existence.   A direct consequence of this result is the global well-posedness  of   two-dimensional Euler solutions for smooth initial data since the vorticity is only advected and then  does not grow.   We emphasize that new geometric blowup criteria are recently discovered by Constantin, Fefferman and Majda \cite{CFM}.   

Let us recall that, in space dimension three, the vorticity is  defined by the vector $\omega={\rm curl}\,  u$ and satisfies the equation 
$$
\partial_t \omega+(u\cdot\nabla)\omega-(\omega \cdot\nabla)u=0.  
$$
 The main 
difficulty for establishing global regularity is to understand how the vortex stretching term
$(\omega \cdot\nabla)u$ affects the dynamic of the fluid.   

\

While global existence is not proved for arbitrary initial smooth data, there are partial results in the case of the so-called axisymmetric flows without swirl.  
We say that a  vector field $u$ is  axisymmetric  if it has the form:
$$
u(x, t) = u^r(r, z , t)e_r + u^z (r, z , t)e_z,\quad x=(x_{1},x_{2},z),\quad r=({x_{1}^2+x_{2}^2})^{\frac12},
$$ 
where  $\big(e_r, e_{\theta} , e_z\big)$ is the cylindrical basis of $\mathbb R^3$ and the components $u^r$ and $u^z$ do not depend on the angular variable.   The main feature  of axisymmetric  flows   arises in the  
vorticity  which takes the form (more precise discussion will be done in Proposition \ref{mim0} and \ref{mim}),
$$
\omega=(\partial_zu^r-\partial_ru^z) e_{\theta}
$$
and satisfies 
  \begin{equation}
 \label{tourbillon}
\partial_t \omega +(u\cdot\nabla)\omega =\frac{u^r}{r}\omega.  
\end{equation}
Consequently the quantity $\alpha:=\omega/r$ is only advected by the flow, that is 
\begin{equation}
\label{equation_importante}
\partial_t\alpha+(u\cdot\nabla)\alpha=0.  
\end{equation}
This fact induces the conservation of all the norms  $\|\alpha\|_{L^p}, 1\leq p\leq\infty.  $ \mbox{ In \cite{Ukhovskii}, }
  Ukhovskii and Yudovich took advantage of these conservation laws  
to prove the global existence for axisymmetric initial data with finite energy and satisfying in addition  $\omega_{0}\in L^2\cap L^\infty$ and $\frac{\omega_{0}}{r}\in L^2\cap L^\infty.  $ In terms of Sobolev regularity these assumptions are satisfied if the velocity   $u_{0}$ belongs to $ H^s$ with $s>{7\over 2}$.   This is far from the critical regularity of local existence theory $s=\frac52.  $ The optimal result in Sobolev spaces is done  by Shirota and Yanagisawa  in \cite{Taira} who proved global existence in $H^s,$ with  $s>\frac52$.  
Their proof is based on the boundness of the quantity $\|\frac{u^r}{r}\|_{L^\infty}$ by using Biot-Savart law.   We mention also the reference  \cite{Saint} where  similar results  are given in different function spaces.   In a recent work \cite{rd}, Danchin has weakened the Ukhoviskii and  Yudovich conditions.   More precisely, he obtains global existence and uniqueness for initial data  $\omega_0\in L^{3,1}\cap L^\infty$ and $\frac{\omega_0}{r}\in L^{3,1}.   $ Here, we denote by  $L^{3,1}$ the Lorentz space.  

\

 In this paper we address the question of global existence in the critical \mbox{spaces $B_{p,1}^{1+3/p}$.}    Comparing to the sub-critical spaces this problem is extremely hard to deal with   because we are deprived of an important tool which is the BKM criterion.   Even in space dimension two we encounter the same problem.   Although the \mbox{quantity $\|\omega(t)\|_{L^\infty}$} is conserved, this is not sufficient to propagate for all time the initial regularity.   As it was pointed by Vishik in \cite{vishik} the significant quantity  is $\|\omega(t)\|_{B_{\infty,1}^0}$ and its control needs the use of  the special structure of the vorticity, which is only transported by the flow.   
 
 \
 
 Owing to the streching term $\omega\, u^r/r$, the estimate  of $\|\omega(t)\|_{B_{\infty,1}^0}$ for axisymmetric flows is more complicated and needs as we shall see a  refined analysis of the geometric structure of the vorticity.   

The main result of this paper can  be stated as follows (for the definition of function spaces see next section).  
\begin{theo}
\label{thm0}  Assume  $p\in[1,\infty]$.     Let $u_{0}$ be an axisymmetric divergence free vector field belonging to  $ B_{p,1}^{1+3/p},$ such that  its vorticity satisfies $\displaystyle{{\omega_{0}}/{r}\in L^{3,1}}$.     Then the system \eqref{E} has a unique global solution $
 u\in{\mathcal C}(\mathbb R_
 +;\,  B^{1+3/p}_{p,1}).    $

\end{theo}
\begin{rema}
\label{RK1}
 We mention that for $p<3$ the condition $\frac{\omega_{0}}{r}\in L^{3,1}$ is automatically derived from $u_{0}\in B_{p,1}^{1+3/p}$ (see Proposition \ref{lorentz} below).    
 \end{rema}
\begin{rema}
\label{RK}
In the proof of  this theorem we have established the following global in time   estimates
 \begin{equation*}
 \label{controle}
\|u(t)\|_{ B^{s_p}_{p,1}}
\leq C_{0}e^{e^{\exp{C_{0}t}}}\quad\hbox{and}\quad \|{\omega}(t)/{r}\|_{L^{3,1}}\leq\|{\omega_{0}}/{r}\|_{L^{3,1}},
\end{equation*}
 where the  \mbox{constant $C_{0}$} depends only on the initial data norms.    
 \end{rema}

The proof is heavily related to two crucial estimates, the first  one is the $L^\infty $ bound of the vorticity for every time which is obtained from Biot-Savart law and the use of Lorentz spaces, see Proposition \ref{u/r}.   Unfortunately as it has been discussed above this is not sufficient to show global existence because we do not know whether the BKM criterion  works in the critical spaces or not.   Thus we are led to establish a second new estimate for the vorticity in  Besov \mbox{space $B_{\infty,1}^0$} (see Proposition \ref{ur1}).This allows us to bound for every time the Lipschitz norm of the velocity which is  sufficient to prove global existence.The control of $\|\omega(t)\|_{B_{\infty,1}^0}$ is the most important part of this paper and it is done in a non fashion way in which the axisymmetric geometry plays a key role.  

\

 \indent The rest of this paper is organized as follows. In section $ 2$   we recall some function spaces and gather some  preliminary estimates.   Section $3$ is devoted to the study of some geometric properties of any solution to a vorticity equation model.    The  proof  of  
\mbox{Theorem \ref{thm0}} is   done in several steps in \mbox{section 4.  } 
 
\section{Notations and preliminaries} 
\hskip 12pt
Throughout this paper, $C$ stands for some real positive constant which may be different in each occurrence.   We shall sometimes alternatively use the notation $X\lesssim Y$ for an inequality of type $X\leq CY$.  

\

$\bullet $ Let us start with  a classical dyadic decomposition of the full space (see for instance \cite{Ch1}):
there exist two radial  functions  $\chi\in \mathcal{D}(\mathbb R^3)$ and  
$\varphi\in\mathcal{D}(\mathbb R^3\backslash{\{0\}})$ such that
\begin{itemize}
\item[\textnormal{i)}]
$\displaystyle{\chi(\xi)
+\sum_{q\geq0}\varphi(2^{-q}\xi)=1}$\; 
$\forall\xi\in\mathbb R^3,$
\item[\textnormal{ii)}]
$\displaystyle{\sum_{q\in\mathbb Z}\varphi(2^{-q}\xi)=1}$\;
if\; $\xi\neq 0,$
\item[\textnormal{iii)}]
$ \textnormal{supp }\varphi(2^{-p}\cdot)\cap
\textnormal{supp }\varphi(2^{-q}\cdot)=\varnothing,$ if  $|p-q|\geq 2$,\\

\item[\textnormal{iv)}]
$\displaystyle{q\geq1\Rightarrow \textnormal{supp}\chi\cap \textnormal{supp }\varphi(2^{-q})=\varnothing}$.  
\end{itemize}
For every $u\in{\mathcal S}'(\mathbb R^3)$ one defines the nonhomogeneous Littlewood-Paley operators by,
$$
\Delta_{-1}u=\chi(\hbox{D})u;\, \forall
q\in\mathbb N,\;\Delta_qu=\varphi(2^{-q}\hbox{D})u\; \quad\hbox{and}\quad
S_qu=\sum_{-1\leq j\leq q-1}\Delta_{j}u.  
$$
One can easily prove that for every tempered distribution $u,$
\begin{equation}\label{dr2}
u=\sum_{q\geq -1}\Delta_q\,u.  
\end{equation}
The homogeneous operators are defined as follows
$$
\forall q\in\mathbb Z,\,\quad\dot{\Delta}_{q}u=\varphi(2^{-q}\hbox{D})v\quad\hbox{and}\quad\dot{S}_{q}u=\sum_{j\leq q-1}\dot{\Delta}_{j}u.   
$$
We notice that these operators can be written as a convolution.   For example for $q\in\ZZ,\,$ $\dot\Delta_{q}u=2^{3q}h(2^q\cdot)\star u,$ where $h\in\mathcal{S}$ and $\widehat{h}(\xi)=\varphi(\xi).  $

For the homogeneous decomposition, the identity (\ref{dr2}) is not true due to the polynomials but we have,
$$
u=\sum_{q\in\mathbb Z}\dot\Delta_qu
\quad\forall\,u\in 
{\mathcal {S}}'(\mathbb R^3)/{\mathcal{P}}[\mathbb R^3],
$$
where ${\mathcal{P}}[\mathbb R^3]$ is the whole of 
polynomials (see \cite{PE}).

$\bullet $
In the sequel we will  make an extensive use of Bernstein inequalities (see for \mbox{example \cite{Ch1}}).  
\begin{lemm}\label{lb}\;
 There exists a constant $C$ such that for $k\in\NN$, \mbox{$1\leq a\leq b$}   and $\psi\in L^a$, we have
\begin{eqnarray*}
\sup_{|\alpha|=k}\|\partial ^{\alpha}S_{q}\psi\|_{L^b}&\leq& C^k\,2^{q(k+d(\frac{1}{a}-\frac{1}{b}))}\|S_{q}\psi\|_{L^a},
\end{eqnarray*}
and
\begin{eqnarray*}
\ C^{-k}2^
{qk}\|\dot{\Delta}_{q}\psi\|_{L^a}&\leq&\sup_{|\alpha|=k}\|\partial ^{\alpha}\dot{\Delta}_{q}\psi\|_{L^a}\leq C^k2^{qk}\|\dot{\Delta}_{q}\psi\|_{L^a}.  
\end{eqnarray*}
\end{lemm}

Let us now introduce the basic tool of  the paradifferential calculus which is  Bony's decomposition  \cite{b}.   It distinguishes in   a product 
$uv$  three parts as follows: 
$$
uv=T_u v+T_v u+\mathcal{R}(u,v),
$$
where
\begin{eqnarray*}
T_u v=\sum_{q}S_{q-1}u\Delta_q v, \quad\hbox{and}\quad \mathcal{R}(u,v)=
\sum_{q}\Delta_qu \widetilde \Delta_{q}v,
\end{eqnarray*}
$$
\textnormal{with}\quad {\widetilde \Delta}_{q}=\sum_{i=-1}^{1}\Delta_{q+i}.  
$$
$T_{u}v$ is called paraproduct of $v$ by $u$ and  $\mathcal{R}(u,v)$ the remainder term.   

 Let $(p,r)\in[1,+\infty]^2$ and $s\in\mathbb R,$ then the nonhomogeneous  Besov 
\mbox{space $B_{p,r}^s$} is 
the set of tempered distributions $u$ such that
$$
\|u\|_{B_{p,r}^s}:=\Big( 2^{qs}
\|\Delta_q u\|_{L^{p}}\Big)_{\ell^{r}}<+\infty.  
$$
We remark that we have the identification  $B_{2,2}^s=H^s$.   Also, by using the Bernstein inequalities  we get easily 
$$
B^s_{p_1,r_1}\hookrightarrow
B^{s+3({1\over p_2}-{1\over p_1})}_{p_2,r_2}, \qquad p_1\leq p_2\quad and \quad  r_1\leq r_2.  
$$

\

$\bullet $  For a measurable function $f$ we define its nonincreasing rearrangement  by
$$
f^{\ast}(t):= \inf\Big\{s,\;\mu\big(\{ x,\; |f (x)| > s\}\big)\leq t\Big\},
$$
where $\mu$ denotes the usual Lebesgue measure.  
For $(p,q)\in [1,+\infty]^2,$ the Lorentz space $L^{p,q}$ is the set of  functions $f$ 
such that $\|f\|_{L^{p,q}}<\infty,$ with
$$
\|f\|_{L^{p,q}}:=\left\lbrace
\begin{array}{l}\displaystyle
\Big(\int_0^\infty[t^{1\over p}f^{\ast}(t)]^q{dt\over t}\Big)^{1\over q}, \quad \hbox{for}\;1\leq q<\infty\\
\displaystyle\sup_{t>0}t^{1\over p}f^{\ast}(t),\quad \hbox{for}\;q=\infty.  
\end{array}
\right.  
$$
Notice that we can also define Lorentz spaces  by real interpolation from Lebesgue spaces:
$$
(L^{p_0},L^{p_1})_{(\theta,q)}=L^{p,q},
$$
where
$
1\leq p_0<p<p_1\leq\infty,
$
$\theta$ satisfies ${1\over p}={1-\theta\over p_0}
+{\theta\over p_1}$ and $1\leq q\leq\infty$.  
We have the classical properties:
\begin{equation}\label{imbed0}
\|uv\|_{L^{p,q}}
\le
\|u\|_{L^\infty}\|v\|_{L^{p,q}}.  
\end{equation}
\begin{equation}
\label{imbed23}L^{p,q}\hookrightarrow L^{p,q'},\forall\, 1\leq p\leq\infty; 1\leq q\leq q'\leq \infty\quad \hbox{and}\quad L^{p,p}=L^p.  
\end{equation}

\

The next proposition precises the statement of Remark \ref{RK1}.  
\begin{prop}\label{lorentz}
Let $1\leq p<3$  and $u\in B_{p,1}^{1+3/p}(\mathbb R^3)$ be an axisymmetric divergence free vector field.   If we denote by $\omega$ its vorticity, then we have
$$
\|{\omega}/{r}\|_{L^{3,1}}\lesssim \|u\|_{B_{p,1}^{1+3/p}}.  
$$
\end{prop}
\begin{proof}
First we start with showing the embedding  $B_{p,1}^{\frac{3}{p}-1}\hookrightarrow L^{3,1}.$  Let $(p,r)$ be a fixed   exponents pair   satisfying $
1\leq p<3<r\leq\infty
$.   By definition, we have
 $$
(L^{p},L^{r})_{(\theta,1)}=L^{3,1},\quad\hbox{
with }\quad{1\over3}=\frac{1-\theta}{ p}
+{\theta\over r}\cdot
$$
According to  Bernstein inequalities, we have
$$
B^{{3\over p}-{3\over r}}_{p,1}
\hookrightarrow B^0_{r,1}\hookrightarrow L^{r}
\quad\mbox{and}\quad
B^0_{p,1}\hookrightarrow L^{p}.  
$$
Consequently,
$$
(B^0_{p,1},B^{{3\over p}-{3\over r}}_{p,1})_{(\theta,1)}
\hookrightarrow L^{3,1}.  
$$
On the other hand, we have    (see for instance \cite{BE} page 152), 
$$
(B^0_{p,1},B^{{3\over p}-{3\over r}}_{p,1})_{(\theta,1)}
=B^{\theta({3\over p}-{3\over r})}_{p,1}
=B^{{3\over p}-1}_{p,1}.  
$$
This completes the proof of the embedding  $B_{p,1}^{\frac{3}{p}-1}\hookrightarrow L^{3,1}.  $
 From this it ensures
\begin{eqnarray*}
\|\nabla \omega\|_{L^{3,1}}&\lesssim& \|\nabla\omega\|_{B_{p,1}^{\frac{3}{p}-1}}
\\
&\lesssim&\|u\|_{B_{p,1}^{1+3/p}}.  
\end{eqnarray*}
It remains  to show the following estimate
$$
\|\omega/r\|_{L^{3,1}}\lesssim\|\nabla\omega\|_{L^{3,1}}.  
$$
Since $B_{p,1}^{\frac{3}{p}}  \hookrightarrow B_{\infty,1}^{0}  \hookrightarrow  C^0$, then $\omega$ is a continuous function and    Proposition \ref{mim0} implies  $\omega(0,0,z)=0.  $  By a standard smoothing procedure we may assume that $\omega$ is sufficiently smooth.  
According to Taylor formula we write 
$$
\omega(x_{1},x_{2},z)=\int_{0}^1\Big(x_{1}\partial_{x_{1}}\omega(\tau x_{1},\tau x_{2},z)+x_{2}\partial_{x_{2}}\omega(\tau x_{1},\tau x_{2},z)\Big)d\tau.  
$$
Therefore we obtain from (\ref{imbed0}) and by homogeneity
\begin{eqnarray*}
\|\omega/r\|_{L^{3,1}}&\lesssim&\int_{0}^1\|\nabla\omega(\tau\cdot,\tau\cdot,\cdot)\|_{L^{3,1}}d\tau\\
&\lesssim& \|\nabla\omega\|_{L^{3,1}}\int_{0}^1\tau^{-\frac{2}{3}}d\tau\\
&\lesssim &\|\nabla\omega\|_{L^{3,1}}.  \end{eqnarray*}
This achieves the proof.  
\end{proof}

\

The following result will be needed.  
\begin{prop}\label{LZ}
Given $(p,q)\in[1,\infty]^2$ and a smooth divergence free vector field $u.  $    Let $f$ be a smooth solution  of the transport equation 
 $$
\partial_{t}f+u\cdot\nabla f=0,\, f_{|t=0}=f_0.  
$$
Then we have 
$$
\|f(t)\|_{L^{p,q}}\leq\|f_{0}\|_{L^{p,q}}.  
$$
\end{prop}
\begin{proof}
We use the conservation of  the Lebesgue norms combined with a standard interpolation argument, see for instance \cite{BE}.  
\end{proof}
\section{Geometric properties of the vorticity }
\hskip 12pt
In this section we will describe some special geometric properties of axisymmetric flows.   
The following is classical and for the convenience of the reader we give the proof.  
\begin{prop}
\label{mim0}
 Let $u=(u^1,u^2,u^3)$ be a smooth axisymmetric vector field.   
Then we have
\begin{itemize}
\item[i)]  the vector $\omega=\nabla\times u=(\omega^1,\omega^2,\omega^3)$ satsifies $\omega\times e_{\theta}=(0,0,0).  $
 In particular, we have for every $(x_{1},x_{2},z)\in \mathbb R^3$,
$$
\omega^3=0,\quad x_{1}\omega^1(x_{1},x_{2},z)+x_{2}\omega^2(x_{1},x_{2},z)=0\quad\hbox{and}
$$
$$
 \omega^1(x_{1},0,z)=\omega^2(0,x_{2},z)=0.  
$$
\item[ii)] for every $q\geq -1$,  $\Delta_{q}u$ is axisymmetric and
$$
(\Delta_{q}u^1)(0,x_{2},z)=(\Delta_{q}u^2)(x_{1},0,z)=0.  
$$
\end{itemize}
\end{prop}
\begin{proof}
{i)} The first point is easy to show since
the vorticity is given by 
$$
\omega=\left
(
     \begin{array}{c}
        \sin\theta\partial_{r}u^3-\sin\theta\partial_{3}u^r\\
       \cos\theta\partial_{3}u^r-\cos\theta\partial_{r}u^3\\
       0
     \end{array}
  \right)=(\partial_{3}u^r-\partial_{r}u^3)e_{\theta}.  
  $$
  The other properties are a direct consequence from this information.  
  
  \
  
 ii) To prove that the  angular component $(\Delta_{q}u)^\theta$ is zero it suffices to show that
  $
  x_{2}\Delta_{q}u^1-x_{1}\Delta_{q}u^2=0.  
  $
  In Fourier variables it is equivalent to the identity $\partial_{\xi_{2}}(\widehat{\Delta_{q}u^1})-\partial_{\xi_{1}}(\widehat{\Delta_{q}u^2})=0.  $ Since $\varphi$ is radial then we get
  \begin{eqnarray*}
  \partial_{\xi_{2}}(\widehat{\Delta_{q}u^1})-\partial_{\xi_{1}}(\widehat{\Delta_{q}u^2})&=&\varphi(2^{-q}|\xi|)(\partial_{\xi_{2}}\widehat{u^1}-\partial_{\xi_{1}}\widehat{u^2})\\
  &+&2^{-q}|\xi|^{-1}\varphi'(2^{-q}|\xi|)\big(\xi_{2}\widehat{u^1}-\xi_{1}\widehat{u^2}    \big).  
  \end{eqnarray*}
  Since $u^\theta=0$ then $x_{1}u^2-x_{2}u^1=0,$ and consequently the first term  of the right hand side is zero.   Thus we find
  \begin{eqnarray*}
  \partial_{\xi_{2}}(\widehat{\Delta_{q}u^1})-\partial_{\xi_{1}}(\widehat{\Delta_{q}u^2})&=&
  2^{-q}|\xi|^{-1}\varphi'(2^{-q}|\xi|)\big(\xi_{2}\widehat{u^1}-\xi_{1}\widehat{u^2}    \big)\\
  &=&-i2^{-q}|\xi|^{-1}\varphi'(2^{-q}|\xi|)\widehat{\omega^3}\\
  &=&0.  
  \end{eqnarray*}
  It follows that $\Delta_{q}u=(\Delta_{q}u)^re_{r}+\Delta_{q}u^3e_{z}.  $ To end the proof it remains to show that both components do not depend on the angle $\theta.  $ For this purpose it suffices to have 
  $$
  (\Delta_{q}u)^r(\mathcal{R}_{\eta}x)=(\Delta_{q}u)^r(x)\quad\hbox{and}\quad  \Delta_{q}u^3(\mathcal{R}_{\eta}x)=\Delta_{q}u^3(x),
  $$
  where $\mathcal{R}_{\eta}$ is the rotation with angle $\eta$ and axis $(Oz).  $ It is easy to see from the definition  that
  $$
  (\Delta_{q}u)^r(x)=2^{3q}\int_{\RR^3}h(2^q(x-y))u^r(y)\,e_{r}(y)\cdot e_{r}(x)dy.  
  $$
  Since $h$ is radial and the rotation preserves angles and distances then we get
  \begin{eqnarray*}
   (\Delta_{q}u)^r(\mathcal{R}_{\eta}x)&=&2^{3q}\int_{\RR^3}h(2^q\mathcal{R}_{\eta}(x-y))u^r(\mathcal{R}_{\eta}y)\,e_{r}(\mathcal{R}_{\eta}y)\cdot e_{r}(\mathcal{R}_{\eta}x)dy\\
   &=&2^{3q}\int_{\RR^3}h(2^q(x-y))u^r(\mathcal{R}_{\eta}y)\,e_{r}(y)\cdot e_{r}(x)dy\\
   &=&
   (\Delta_{q}u)^r(x).  
  \end{eqnarray*}
  We have used in the last line the fact that $u^r(\mathcal{R}_{\eta}y)=u^r(y)$ which is an easy consequence of the axisymmetry of the flow.   By the same way we obtain $\Delta_{q}u^3(\mathcal{R}_{\eta}x)=\Delta_{q}u^3(x).  $ This achieves the proof.  
\end{proof}

\

 The last  part of this section is dedicated to the  study a vorticity equation type in which no relations between the vector field $u$ and the solution $\Omega$ are supposed.   More precisely, we consider
\begin{equation}
\label{V}
\left\lbrace
\begin{array}{l}
\partial_t \Omega+(u\cdot\nabla)\Omega=\Omega\cdot\nabla u,\\
\diver u=0,\\
{\Omega}_{| t=0}=\Omega_0.  
\end{array}
\right.  
\end{equation}
We will assume that $u$ is axisymmetric and the unknown function  $\Omega=(\Omega^1,\Omega^2,\Omega^3)$ is a vector field.   The following result describes the preservation   of some  initial geometric conditions of the solution $\Omega.  $
\begin{prop}\label{mim}
Let $u$ be  a divergence free and axisymmetric vector field belonging to  $L^1_{\textnormal loc}(\RR_{+}, \textnormal{Lip}(\RR^3))$ and  $\Omega $ the unique global solution  of \eqref{V} with smooth initial data $\Omega_{0}.  $ Then the following properties hold.  
\begin{itemize}
\item[i)] If $\diver \Omega_{0}=0$ then $\diver \Omega(t)=0$, for every $t\in\RR_+.  $
\item[ii)] If $\Omega_{0}\times e_{\theta}=(0,0,0)$ then we have 
$$
\qquad \Omega(t)\times e_{\theta}=(0,0,0),\qquad  \forall\, t\in\RR_+.  
$$  
Consequently, $\Omega^1(t,x_{1},0,z)=\Omega^2(t,0,x_{2},z)=0,$ and 
$$
\partial_t \Omega+(u\cdot\nabla)\Omega= \frac{u^r}{r}\Omega.  
$$
\end{itemize}
\end{prop}
\begin{proof}
First, we notice that the existence and uniqueness of global solution can be done in classical way.   Indeed, let $\psi$ denote the flow of the velocity $u,$ that is the vector-valued function satisfying 
$$
\psi(t,x)=x+\int_{0}^tu(\tau,\psi(\tau,x))d\tau.  
$$
Since $u\in L^1_{\textnormal loc}(\RR_{+}, \textnormal{Lip}(\RR^3))
$ then it follows from the ODE theory that the function  $\psi$ is uniquely and globally defined.   Let $\widetilde\Omega(t,x):=\Omega(t,\psi(t,x))$ and $A(t,x)$ the matrix such that $A(t,\psi^{-1}(t,x))=(\partial_{j}u_{i})_{1\leq i,j\leq 3},$ then it is obvious that 
$$
\partial_{t}\widetilde\Omega=A(t,x)\widetilde\Omega.  
$$
From Cauchy-{Lipschitz } theorem this last equation  has a unique global solution, and the system \eqref{V} too.  

\

i) We apply the divergence operator to the equation \eqref{V} leading under the assumption $\diver u=0,$ to
$$
\partial_{t}\diver\Omega+u\cdot\nabla\diver\Omega=0.  
$$
Thus, the quantity $\diver\Omega$ is transported  by the flow and consequently the incompressibility of $\Omega$ remains true for every time.

\

ii) We denote by $(\Omega^r,\Omega^\theta,\Omega^z)$ the coordinates of $\Omega$ in cylindrical basis.   It is obvious that 
$\Omega^r=\Omega\cdot e_r.  $ Recall that in cylindrical coordinates the \mbox{operator $u\cdot\nabla$} has the form
$$
u\cdot\nabla=u^r\partial_{r}+\frac1r u^\theta\partial_{\theta}+u^z\partial_{z}=u^r\partial_{r}+u^z\partial_{z}.  
$$
We have used in the last equality the fact that for axisymmetric flows the angular component is zero.   Hence we get
\begin{eqnarray*}
(u\cdot\nabla\Omega)\cdot e_{r}&=&u^r\partial_{r}\Omega\cdot e_{r}+u^z\partial_{z}\Omega\cdot e_{r}\\
&=&(u^r\partial_{r}+u^z\partial_{z})(\Omega\cdot e_{r})\\
&=&u\cdot\nabla \Omega^r,
\end{eqnarray*}
Where we use $\partial_{r}e_{r}=\partial_{z}e_{r}=0.  $
 Now it remains to compute $(\Omega\cdot\nabla u)\cdot e_{r}.  $
 By a straightforward computations we get,
  \begin{eqnarray*}
 (\Omega\cdot\nabla u)\cdot e_{r}&=&\Omega^{r}\,\partial_{r}u\cdot e_{r}+\frac1r\Omega^{\theta}\,\partial_{\theta}u\cdot e_{r}+\Omega^3\,\partial_{3}u\cdot e_{r}\\
 &=&\Omega^{r}\partial_{r}u^r+\Omega^3\partial_{3}u^r.  
 \end{eqnarray*}
Thus the component $\Omega^r$ obeys to the equation
$$
\partial_{t}\Omega^r+u\cdot \nabla \Omega^r=\Omega^r\partial_{r}u^r+\Omega^3
\partial_{3}u^r.  
$$
From the maximum principle  we deduce
$$
\|\Omega^r(t)\|_{L^\infty}\leq\int_{0}^t\big(\|\Omega^r(\tau)\|_{L^\infty}+\|\Omega^3(\tau)\|_{L^\infty}\big)\|\nabla u(\tau)\|_{L^\infty}d\tau.  
$$

On the other hand the component $\Omega^3$ satisfies the equation
\begin{eqnarray*}
\partial_{t}\Omega^3+u\cdot\nabla\Omega^3&=&\Omega^3\partial_{3}u^3+\Omega^r\partial_{r}u^3.  
\end{eqnarray*}
 This leads to
 $$
 \|\Omega^3(t)\|_{L^\infty}\leq \int_{0}^t\big(\|\Omega^3(\tau)\|_{L^\infty}+\|\Omega^r(\tau)\|_{L^\infty}\big)\|\nabla u(\tau)\|_{L^\infty}d\tau.  
 $$
 Combining these estimates and using Gronwall's inequality we obtain for every $t\in\RR_{+},$ $\Omega^3(t)=\Omega^r(t)=0,$  which is the desired result.  
 
 Under these assumptions the stretching term becomes
 \begin{eqnarray*}
 \Omega\cdot \nabla u=\frac 1r\Omega^\theta\partial_{\theta}(u^r e_{r})
&=&\frac1r u^r\Omega^\theta e_{\theta}
 \\
 &=&\frac1r u^r\Omega,
 \end{eqnarray*}
 which ends the proof of Proposition \ref{mim}.  
 \end{proof}

\section{Proof of Theorem \ref{thm0}}
\hskip 12pt
The proof of Theorem \ref{thm0}  will be done in several steps and it suffices to establish the {\it a priori} estimates.   The hard part of the proof will be the Lipschitz bound of the velocity.  
 \subsection{Some a priori estimates}
We start with the following
\begin{prop} 
\label{u/r}
Let $u$ be an axisymmetric solution of \eqref{E}, then we have for every $t\in\RR_{+},$
\begin{itemize}
\item[i)]
Biot-Savart law:
$$
\big\|{u^r(t)/r}\big\|_{L^\infty}
\lesssim
\big\|{\omega_0/r}\big\|_{L^{3,1}}.  
$$
\item[ii)] Vorticity bound:
$$
\|\omega(t)\|_{L^\infty}\lesssim \|\omega_{0}\|_{L^\infty}e^{Ct\|\omega_0/r\|_{L^{3,1}}}.  
$$
\item[iii)] Velocity bound:
$$
\|u(t)\|_{L^\infty}\lesssim\big(\|u_{0}\|_{L^\infty}+\|\omega_{0}\|_{L^\infty} \big) e^{\exp{Ct\|\omega_0/r\|_{L^{3,1}}}}.  
$$
\end{itemize}
\end{prop}
\begin{proof}

i) According to  Lemma 1 in \cite{Taira} (see also \cite{rd}) one has, for every $x=(x_1,x_2,x_3)\in\RR^3$,
$$
|u^r(t,x)|
\lesssim \int_{| y-x|\leq r}\frac{|\omega(t,y)|}{|x-y|^{2}}dy
+r \int_{| y-x|\geq r}\frac{|\omega(t,y)|}{|x-y|^{3}}dy,
$$
with $r=({x_{1}^2+x_{2}^2})^{\frac12}$.  
Thus, if we denote $r'=({y_1^2+y_2^2})^{\frac12}$, one can estimate
$$
\begin{aligned}
|u^r(t,x)|
\lesssim
\int_{| y-x|\leq r}{\frac{|\omega(t,y)|}{r'}}{r'\over |x-y|^2}dy
+r \int_{| y-x|\geq r}{|\omega(t,y)|\over r'}{r'-r+r\over |x-y|^3}dy.  
\end{aligned}
$$
But since  $|r'-r|\leq |x-y|$ we have 
$$
|x-y|\leq r\Rightarrow r'\leq 2r.  
$$ 
This yields in particular 
$$
\begin{aligned}
|u^r(t,x)|
&\lesssim
 r\int_{\mathbb R^3}{|\omega(t,y)|\over r'}{1\over |x-y|^2}dy,
\end{aligned}
$$
which can be rewritten as
$$
|u^r/r|\lesssim {1\over |\cdot|^2}\star |\omega/r|.  
$$
 As  ${1\over |\cdot|^2}\in L^{{3\over 2},\infty}(\mathbb R^3)$, 
 then Young inequalities on  $L^{q,p}$ spaces\footnote{The convolution 
$L^{p,q}\star L^{p',q'} \longrightarrow L^{\infty}$
is a bilinear  continuous operator, (see \cite{Neil}, page 141 for more details).  } 
  imply
$$
\big\|{u^r/ r}\big\|_{L^\infty}
\lesssim \big\|{\omega/ r}\big\|_{L^{3,1}}.  
$$
Since  $\omega/ r$  satisfies \eqref{equation_importante} then applying Proposition \ref{LZ} gives
\begin{equation*}
\label{yazid}
\big\|{u^r/ r}\big\|_{L^\infty}
\lesssim \big\|{\omega_{0}/ r}\big\|_{L^{3,1}}.  
\end{equation*}

\

ii) From the maximum principle applied to  \eqref{tourbillon}  one has
$$
\|\omega(t)\|_{L^\infty}
\leq
\|\omega_0\|_{L^\infty}
+\int_0^t\big\|{u^r(\tau)/ r}\big\|_{L^\infty}\|\omega(\tau)\|_{L^\infty}d\tau.  
$$
Using  Gronwall's lemma and i) gives the desired result.  

\

iii)
To estimate $L^\infty$ norm of the velocity we use the argument of Serfati \cite{Serfaty},
$$
\begin{aligned}
\|u(t)\|_{L^\infty}
&\leq \|\dot S_{-N}u\|_{L^\infty}+
\sum_{q\geq - N}\|\dot \Delta_qu\|_{L^\infty},
\end{aligned}
$$
where $N$ is an arbitrary positive integer that will be fixed later.  
By Bernstein inequality we infer\footnote{ We recall the classical fact $\|\dot\Delta_q u\|_{L^p}\approx 2^{-q}\|\dot\Delta_q \omega\|_{L^p}$ uniformly in $q$, for every $p\in [1,+\infty]$.  }
 $$
\sum_{q\geq -N}\|\dot\Delta_q u\|_{L^\infty}
\lesssim 
2^{N}\|\omega\|_{L^\infty}.  
$$
On the other hand using the integral equation  we get
$$
\begin{aligned}
\|\dot S_{-N} u\|_{L^\infty}
&\leq 
\|\dot S_{-N}u_0\|_{L^\infty}
+
\int_0^t\|\dot S_{-N}\big(\mathbb{P}(u\cdot\nabla)u\big)\|_{L^\infty}d\tau
\\&
\lesssim
\| u_0\|_{L^\infty}
+\sum_{j<-N}\int_0^t\|\dot \Delta_j\big(\mathbb{P}(u\cdot\nabla)u\big)\|_{L^\infty}d\tau,
\end{aligned}
$$
where $\mathbb{P}$ denotes the Leray's projector over divergence free vector fields.    Since $\dot \Delta_j \mathbb{P}$ maps $L^p$ to itself uniformly\footnote{We stress the fact that this is true for every $p\in [1,+\infty]$ since  $\mathbb{P}$ is a Fourier multiplier of degree zero so $\dot \Delta_j \mathbb{P}= \Psi(2^{-j}D)$, where $\Psi\in C^\infty_0$.  } in $j\in\mathbb Z,$ we get  
$$
\|\dot S_{-N} u\|_{L^\infty}
\lesssim
\| u_0\|_{L^\infty}
+2^{-N}\int_0^t\|u(\tau)\|_{L^\infty}^2d\tau.  
$$
Hence we obtain
$$
\|u(t)\|_{L^\infty}\lesssim\|u_{0}\|_{L^\infty}+2^N\|\omega(t)\|_{L^\infty}+2^{-N}\int_{0}^t\|u(\tau)\|_{L^\infty}^2d\tau.  
$$
If we choose $N$ such that   
$$
2^{2N}\approx 1+\|\omega(t)\|_{L^\infty}^{-1}\int_{0}^t\|u(\tau)\|_{L^\infty}^2d\tau,
$$
then we obtain
$$
\|u(t)\|_{L^\infty}^2
\lesssim
\| u_0\|_{L^\infty}^2+\|\omega(t)\|_{L^\infty}^2
+\|\omega(t)\|_{L^\infty}
\int_0^t\|u(\tau)\|_{L^\infty}^2d\tau.  
$$
Thus Gronwall's  lemma and the $L^\infty$ bound of the vorticity yield
\begin{equation*}
\begin{aligned}
\label{u_infty}
\|u(t)\|_{L^\infty}
&\lesssim
\big(\|u_0\|_{L^\infty}+\|\omega\|_{L^\infty_{t}L^\infty}\big)e^{Ct\|\omega\|_{L^\infty_{t}L^\infty}}
\\&
\lesssim\big(\|u_0\|_{L^\infty}+\|\omega_{0}\|_{L^\infty}\big)e^{\exp{Ct\|\frac{\omega_0}{r}\|_{L^{3,1}}}}.  
\end{aligned}
\end{equation*}
\end{proof}
\subsection{Lipschitz estimate of the velocity}
The Lipschitz estimate of the velocity is heavily related to
the following interpolation result which is the heart of this work:
\begin{prop}
\label{Thm89}
There exists a decomposition $(\tilde{\omega}_{q})_{q\geq-1}$ of the vorticity $\omega$ such that
\begin{itemize}
\item[i)] For every $t\in\RR_{+},\,\omega(t,x)=\sum_{q\geq-1}\tilde\omega_{q}(t,x).  $
\item[ii)] For every $t\in\RR_{+}, \diver \tilde\omega_{q}(t,x)=0.  $
\item[iii)] For every $q\geq-1$ we have $\|\tilde\omega_{q}(t)\|_{L^\infty}\leq\|\Delta_{q}\omega_0\|_{L^\infty}e^{Ct\|\omega_0/r\|_{L^{3,1}}}.  $
\item[iv)] For all $j,q\geq-1$ we have
$$
\|\Delta_{j}\tilde\omega_{q}(t)\|_{L^\infty}\leq C2^{-|j-q|}e^{CU(t)}\|\Delta_{q}\omega_0\|_{L^\infty},$$
with $U(t):=\|u\|_{L^1_{t}B_{\infty,1}^1}$ and $C$ an absolute constant.  
\end{itemize}
\end{prop}
\begin{proof}
We will use for this purpose a new approach similar to \cite{ST}.   
Let $q\geq -1$ and  denote by $\tilde{\omega}_{q}$  the unique global vector-valued solution of the problem
\begin{equation}\label{td1}
\left\lbrace
\begin{array}{l}
\partial_t \tilde{\omega}_{q}+(u\cdot\nabla)\tilde{\omega}_{q}=\tilde\omega_{q}\cdot\nabla u\\
{\tilde\omega_q}{_{|t=0}}=\Delta_{q}\omega_{0}.  
\end{array}
\right.  
\end{equation} 
Since $\diver \Delta_{q}\omega_{0}=0,$ then it follows from Proposition \ref{mim} that 
$\diver\tilde\omega_{q}(t,x)=0.  $ On the other hand we have
by linearity and uniqueness
\begin{equation}\label{LU}
\omega(t,x)=\sum_{q\geq-1}\tilde{\omega}_{q}(t,x).  
\end{equation}
We will now rewrite the equation (\ref{td1}) under a suitable form.   

As $\Delta_{q}\omega_{0}=\hbox{curl }\Delta_{q}u_{0}$ and $\Delta_{q}u_{0}$ is axisymmetric then we obtain from Proposition \ref{mim0} that $(\Delta_{q}\omega_{0})\times e_{\theta}=(0,0,0).  $ This leads in view of   Proposition \ref{mim} to $\ww_{q}(t)\times e_{\theta}=(0,0,0)$ and 
\begin{equation}\label{td2}
\left\lbrace
\begin{array}{l}
\partial_t \tilde{\omega}_{q}+(u\cdot\nabla)\tilde{\omega}_{q}=\frac{u^r}{r}\tilde\omega_{q}\\
{\tilde\omega_q}{_{|t=0}}=\Delta_{q}\omega_{0}.  
\end{array}
\right.  
\end{equation}
Applying the maximum principle and using Proposition \ref{u/r} we obtain
\begin{eqnarray}
\nonumber
\|\tilde \omega_{q}(t)\|_{L^\infty}
&\leq& \|\Delta_{q}\omega_{0}\|_{L^\infty}e^{\int_0^t 
\|{u^r(\tau)/r}\|_{L^\infty}d\tau}
\\
\label{ben0}
&\leq&\|\Delta_{q}\omega_{0}\|_{L^\infty}e^{Ct\|\frac{\omega_{0}}{r}\|_{L^{3,1}}}.  
\end{eqnarray}
This concludes the proof of i-iii) of the proposition.  

\

Let us now move to the proof of  iv) which is the main property  of the decomposition above.  
Remark first that the desired estimate is equivalent to 
\begin{eqnarray}
\label{bs10-2}
\|\Delta_{j}\tilde\omega_{q}(t)\|_{L^\infty}\leq C2^{j-q}e^{CU(t)}\|\Delta_{q}\omega_0\|_{L^\infty}
\end{eqnarray}
and
\begin{eqnarray}
\label{bs10-1}
 \|\Delta_{j}\tilde\omega_{q}(t)\|_{L^\infty}\leq C2^{q-j}e^{CU(t)}\|\Delta_{q}\omega_0\|_{L^\infty} .  
\end{eqnarray}

\

$\bullet $ {\sl Proof of  \eqref{bs10-2}.  }
 Applying  Proposition \ref{Lems2} of the appendix  to  \eqref{td1}  
\begin{multline}
\label{bs10}
 e^{-CU(t)}\|\ww_{q}(t)\|_{B_{\infty,\infty}^{-1}}\lesssim\|\Delta_{q}\omega_0\|_{B_{\infty,\infty}^{-1}}+
\\+\int_{0}^te^{-CU(\tau)}\|\ww_{q}\cdot\nabla u(\tau)\|_{B_{\infty,\infty}^{-1}}d\tau.  
\end{multline}
To  estimate the integral term we write in view of Bony's decomposition
\begin{eqnarray*}
\|\ww_{q}\cdot\nabla u\|_{B_{\infty,\infty}^{-1}}&\leq&\|T_{\ww_{q}}\cdot\nabla u\|_{B_{\infty,\infty}^{-1}}+\|T_{\nabla u}\cdot\ww_{q}\|_{B_{\infty,\infty}^{-1}}\\
&+&\|\mathcal{R}\big({\ww_{q}}\cdot\nabla, u\big)\|_{B_{\infty,\infty}^{-1}}\\
&\lesssim&\|\nabla u\|_{L^\infty}\|\ww_{q}\|_{B_{\infty,\infty}^{-1}}+\|\mathcal{R}\big({\ww_{q}}\cdot\nabla, u\big)\|_{B_{\infty,\infty}^{-1}}.  
\end{eqnarray*}
Since $\diver\ww_{q}=0,$ then the remainder term can be treated as follows
\begin{eqnarray*}
\|\mathcal{R}\big({\ww_{q}}\cdot\nabla, u\big)\|_{B_{\infty,\infty}^{-1}}&=&\|\diver\mathcal{R}\big({\ww_{q}}\otimes, u\big)\|_{B_{\infty,\infty}^{-1}}\\
&\lesssim&\sup_{k}\sum_{j\geq k-3
}\|\Delta_{j}\ww_{q}\|_{L^\infty}\|\widetilde\Delta_{j}u\|_{L^\infty}\\
&\lesssim& \|\ww_{q}\|_{B_{\infty,\infty}^{-1}}\|u\|_{B_{\infty,1}^1}.  
\end{eqnarray*}
Il follows that
$$
\|\ww_{q}\cdot\nabla u\|_{B_{\infty,\infty}^{-1}}\lesssim\|u\|_{B_{\infty,1}^1}\|\ww_{q}\|_{B_{\infty,\infty}^{-1}}
$$
Inserting this estimate into (\ref{bs10}) we get
\begin{eqnarray*}
e^{-CU(t)}\|\ww_{q}(t)\|_{B_{\infty,\infty}^{-1}}&\lesssim&\|\Delta_{q}\omega_0\|_{B_{\infty,\infty}^{-1}}\\
&+&\int_{0}^t\|u(\tau)\|_{B_{\infty,1}^1}e^{-CU(\tau)}\|\ww_{q}(\tau)\|_{B_{\infty,\infty}^{-1}}d\tau.                        
\end{eqnarray*}
Hence we obtain by Gronwall's inequality
\begin{eqnarray*}
\|\ww_{q}(t)\|_{B_{\infty,\infty}^{-1}}&\leq& C\|\Delta_{q}\omega_0\|_{B_{\infty,\infty}^{-1}}e^{CU(t)}\\
&\leq& C2^{-q}\|\Delta_{q}\omega_0\|_{L^\infty}e^{CU(t)}.  
\end{eqnarray*}
This gives by definition
$$
\|\Delta_{j}\ww_{q}(t)\|_{L^\infty}\leq C2^{j-q}\|\Delta_{q}\omega_0\|_{L^\infty}e^{CU(t)}.  
$$

\

$\bullet $ {\sl Proof of  \eqref{bs10-1}.  }  Since $w_{q}(t)\times e_{\theta}=(0,0,0)$  the solution $\ww_{q}$ has two components in the cartesian basis, $\tilde{\omega}_{q}=(\tilde{\omega}_{q}^1,\tilde{\omega}_{q}^2,0)$.   The analysis will be exactly the same for both components, so we will deal only with the first one.

From the identity ${{u^r}\over{ r}}={{u^1}\over{x_{1}}}={{u^2}\over{ x_{2}}},$ which is an easy consequence of $u^\theta=0$, it is plain that 
the  \mbox{functions $\tilde{\omega}_{q}^1$} is solution of 
$$
\left\lbrace
\begin{array}{l}
\partial_t \tilde{\omega}_{q}^1+(u\cdot\nabla)\tilde{\omega}_{q}^1=u^2{{\tilde{\omega}_q^1}\over x_{2}},\\
{\tilde\omega_q}^1{_{|t=0}}=\Delta_{q}\omega_{0}^1.  
\end{array}
\right.  
$$
Unfortunately, we are not able to close the estimate in Besov space $B_{\infty,\infty}^1$ due to the invalidity of a commutator estimate in  Proposition \ref{Lems2} for the  limiting case $s=1$.   Nevertherless we will be able to do it for Besov space $B_{\infty,1}^{1}.  $

\

 In view of  \mbox{Proposition \ref{Lems2}} in the appendix we have
\begin{equation}
\label{inv}
e^{-CU(t)}\|\tilde{\omega}_{q}^1(t)\|_{B_{\infty,1}^1}\lesssim \|\tilde{\omega}_{q}^1(0)\|_{B_{\infty,1}^1}+\int_{0}^te^{-CU(\tau)}\Big\|u^2\frac{\ww_{q}^1}{x_{2}}(\tau)\Big\|_{B_{\infty,1}^1}d\tau.  
\end{equation}
To estimate the integral  term we write from Bony's decomposition,
$$
\Big\|u^2\frac{\ww_{q}^1}{x_{2}}\Big\|_{B_{\infty,1}^1}\leq\Big\|T_{\frac{\ww_{q}^1}{x_{2}}}u^2\Big\|_{B_{\infty,1}^1}+
\Big\|T_{u^2}\frac{\ww_{q}^1}{x_{2}}\Big\|_{B_{\infty,1}^1}+
\Big\|\mathcal{R}(u^2,{\ww_{q}^1}/{x_{2}})\Big\|_{B_{\infty,1}^1}.  
$$
To estimate the first paraproduct we write by definition,
\begin{eqnarray}\label{bs1}
\nonumber\Big\|T_{\frac{\ww_{q}^1}{x_{2}}}u^2\Big\|_{B_{\infty,1}^1}&\lesssim& \sum_{j}2^j\|S_{j-1}(\ww_{q}^1/x_{2})\|_{L^\infty}\|\Delta_{j}u^2\|_{L^\infty}\\
&\lesssim&\|u\|_{B_{\infty,1}^1}\|\ww_{q}^1/x_{2}\|_{L^\infty}.  
\end{eqnarray}
The  remainder term is estimated as follows,
\begin{eqnarray}\label{bs2}
\nonumber
\|\mathcal{R}(u^2,\ww_{q}^1/x_{2})\|_{B_{\infty,1}^1}&\lesssim &\sum_{k\geq j-3}2^j\|\Delta_{k}u^2\|_{L^\infty}\|\widetilde\Delta_{k}(\ww_{q}^1/x_{2})\|_{L^\infty}
\\
&\lesssim &\|u\|_{B_{\infty,1}^1}\|\ww_{q}^1/x_{2}\|_{L^\infty}.  
\end{eqnarray}
The treatment of the second term is more subtle and needs  the axisymmetry of the vector field  $u$.   By definition we have
$$
\Big\|T_{u^2}\frac{\ww_{q}^1}{x_{2}}\Big\|_{B_{\infty,1}^1}\lesssim\sum_{j\in\NN}2^j\|S_{j-1}u^2(x)\Delta_{j}(\ww_{q}^1(x)/x_{2})\|_{L^\infty}.  
$$
Now we write
\begin{eqnarray*}
S_{j-1}u^2(x)\Delta_{j}(\ww_{q}^1(x)/x_{2})&=&S_{j-1}u^2(x){\Delta_{j}\ww_{q}^1(x)}/{x_{2}}+S_{j-1}u^2(x)\Big[\Delta_{j}, \frac1x_{2}\Big]\ww_{q}^1\\
&:=&{\hbox{I}}_{j}(x)+{\hbox{II}}_{j}(x).  
\end{eqnarray*}
Since $S_{j-1}u$ est axisymmetric then it follows  from Proposition \ref{mim0} that $S_{j-1}u^2(x_{1},0,z)=0.  $ Thus from Taylor formula we get
$$
\|\hbox{I}_{j}\|_{L^\infty}\lesssim \|\nabla u\|_{L^\infty}\|\Delta_{j}\ww_q^1\|_{L^\infty}.  
$$
This yields
\begin{equation}\label{bs3}
\sum_{j}2^{j}\|\hbox{I}_{j}\|_{L^\infty}\lesssim \|\nabla u\|_{L^\infty}\|\ww_{q}^1\|_{B_{\infty,1}^1}.  
\end{equation}
For the commutator term $\hbox{II}_{j}$ we write by definition
\begin{eqnarray*}
\hbox{II}_{j}(x)&=&S_{j-1}u^2(x)/x_{2}\,\, 2^{3j}\int_{\RR^3}h(2^j(x-y))(x_{2}-y_{2})\ww_{q}^1(y)/y_{2}dy \\
 &=&2^{-j}(S_{j-1}u^2(x)/x_{2})\, 2^{3j}\tilde{h}(2^j\cdot)\star(\ww_{q}^1/y_{2})(x),
\end{eqnarray*}
with $\tilde{h}(x)=x_{2}h(x).  $ Now we claim that for every $f\in\mathcal{S}'$ we have
$$
2^{3j}\tilde{h}(2^j\cdot)\star f=\sum_{|j-k|\leq 1}2^{3j}\tilde{h}(2^j\cdot)\star\Delta_{k}f.  
$$
Indeed, we have  $\widehat{\tilde h}(\xi)=i\partial_{\xi_{2}}\widehat{h}(\xi)=i\partial_{\xi_{2}}\varphi(\xi).  $ It follows that $supp\, \widehat{\tilde h}\subset supp\, \varphi.  $ So we get $2^{3j}\tilde{h}(2^j\cdot)\star\Delta_{k}f=0,$ for $|j-k|\geq 2.  $
This leads to
\begin{eqnarray}\label{bs4}
\nonumber\sum_{j\in\NN}2^j\|\hbox{II}_{j}\|_{L^\infty}&\lesssim&\sum_{|j-k|\leq 1}\|S_{j-1}u^2/x_{2}\|_{L^\infty}\|\Delta_{k}(\ww_{q}^1/x_{2})\|_{L^\infty}\\
&\lesssim& \|\nabla u\|_{L^\infty}\|\ww_{q}^1/x_{2}\|_{B_{\infty,1}^0}.  
\end{eqnarray}
Using (\ref{bs3}) et (\ref{bs4}) one obtains
\begin{equation}
\label{bs5}
\big\|T_{u^2}\ww_{q}^1/x_{2}\big\|_{B_{\infty,1}^1}\lesssim\|\nabla u\|_{L^\infty}\big(\|\ww_{q}^1\|_{B_{\infty,1}^1}+\|\ww_{q}^1/x_{2}\|_{B_{\infty,1}^0}\big).  
\end{equation}
Putting together (\ref{bs1}) (\ref{bs2}) and (\ref{bs5}) we find
$$
\Big\|u^2\frac{\ww_{q}^1}{x_{2}}  \Big\|_{B_{\infty,1}^1}\lesssim \|u\|_{B_{\infty,1}^1}\big(\|\ww_{q}^1\|_{B_{\infty,1}^1}+ \|\ww_{q}^1/x_{2}\|_{B_{\infty,1}^0}\big) 
$$
Therefore we get from (\ref{inv}),
\begin{eqnarray*}
\nonumber e^{-CU(t)}\|\tilde{\omega}_{q}^1(t)\|_{B_{\infty,1}^1}&\lesssim&\|\tilde{\omega}_{q}^1(0)\|_{B_{\infty,1}^1}+
\int_{0}^te^{-CU(\tau)}\|\ww_{q}^1(\tau)\|_{B_{\infty,1}^1}\|u(\tau)\|_{B_{\infty,1}^1}
d\tau\\
&+&\int_{0}^te^{-CU(\tau)}\|u(\tau)\|_{B_{\infty,1}^1}\|{\ww_{q}^1(\tau)}/{x_{2}}\|_{B_{\infty,1}^0}d\tau.  
\end{eqnarray*}
According to Gronwall's inequality we have
\begin{equation}\label{bs6}
\|\tilde{\omega}_{q}^1(t)\|_{B_{\infty,1}^1}\lesssim e^{CU(t)}\Big(\|\tilde{\omega}_{q}^1(0)\|_{B_{\infty,1}^1}+\|{\ww_{q}^1}/{x_{2}}\|_{L^\infty_{t}B_{\infty,1}^0}\Big).  
\end{equation}
Let us now estimate  $\|{\ww_{q}^1}/{x_{2}}\|_{L^\infty_{t}B_{\infty,1}^0}$.   
It is easy to check that  $\tilde{\omega}_{q}^1/x_{2}$ is advected by the flow, that is
$$
\left\lbrace
\begin{array}{l}
(\partial_t +u\cdot\nabla){\tilde{\omega}_{q}^1\over x_{2}}=0\\
{\tilde{\omega}_{q}^1\over x_{2}}{_{|t=0}}={\Delta_{q}\omega_{0}^1\over x_{2}}\cdot
\end{array}
\right.  
$$
Thus we  deduce from  Proposition \ref{Lems2}, 
\begin{equation}\label{rs1}
\big\|{\tilde{\omega}_{q}^1(t)/ x_{2}}\big\|_{B_{\infty,1}^0}
\leq
\big\|{\Delta_q\omega_0^1/x_{2}}\big\|_{B_{\infty,1}^0}e^{CU(t)}.  
\end{equation}

At this stage we need the following  lemma.  
\begin{lemm}Under the assumptions on $u_0$, one has 
 \begin{eqnarray}\label{rs2}
\Big\|{\Delta_q\omega_0^1/x_{2}}\Big\|_{B_{\infty,1}^0}
\lesssim
2^{q}\|\Delta_q\omega_0\|_{L^\infty}.  
\end{eqnarray}
\end{lemm}
\begin{proof}
Since $u_0$ is axisymmetric then according to Proposition \ref{mim0}, $\Delta_q u_0$  is  too.   Consequently $\Delta_q\omega_0$ is the curl of an  axisymmetric vector field and then  by Proposition \ref{mim0} and Taylor expansion
$$
\Delta_q\omega_0^1(x_{1},x_{2},z)=x_{2}\int^1_{0}(\partial_{x_{2}} \Delta_q\omega_0^1)(x_{1},\tau x_{2},z)d\tau.  
$$
 Hence we get in view of Proposition \ref{dilatation}
 \begin{eqnarray*}
\nonumber\big\|{\Delta_q\omega_0^1/x_{2}}\big\|_{B_{\infty,1}^0}
&\leq& \int^1_{0} \Vert( \partial_{x_{2}}\Delta_q\omega_0^1)(\cdot,\tau\cdot,\cdot)\Vert_{B_{\infty,1}^0}d\tau
\\
\nonumber&\lesssim&\|\partial_{x_{2}}\Delta_q\omega_0^1\|_{B_{\infty,1}^0}\int_{0}^1(1-\log\tau)d\tau
\\
&\lesssim& 
2^{q}\| \Delta_q\omega_0^1\|_{L^\infty},
\end{eqnarray*}
as claimed.  
\end{proof}
Coming back to the proof of \eqref{bs10-1}.  
We put together  (\ref{bs6}), (\ref{rs1}) and (\ref{rs2}) to get 
$$
\|\ww_{q}^1(t)\|_{B_{\infty,1}^1}\leq C2^q\|\Delta_{q}\omega_{0}\|_{L^\infty}e^{CU(t)}.  
$$
This can be written as
\begin{equation}\label{bs8}
\|\Delta_{j}\ww_{q}^1(t)\|_{L^\infty}\leq C2^{q-j}e^{CU(t)}\|\Delta_{q}\omega_0\|_{L^\infty},
\end{equation}
which  is \eqref{bs10-1}.  
\end{proof}
In the next proposition we give   some  precise estimates of the velocity.  
\begin{prop}
\label{ur1}  The Euler solution with initial data $u_{0}\in B_{p,1}^{1+\frac3p}$ such \mbox{that $\frac{\omega_{0}}{r}\in L^{3,1}$} satisfies for every  $t\in\RR_{+},$
\begin{itemize}
\item[i)] Case $p=\infty,$
$$
\|\omega(t)\|_{B^0_{\infty,1}}+\| u(t)\|_{B_{\infty,1}^1}
\leq C_{0}{e^{\exp{C_{0}t}}}.  
$$
\item[ii)] Case $1\leq p<\infty,$
$$
\|u(t)\|_{B_{p,1}^{1+\frac{3}{p}}}\leq C_{0}e^{e^{\exp{C_{0}t}}},
$$
with $C_{0}$ depends on the norms of $u_{0}.  $ 
\end{itemize}
\end{prop}
\begin{proof}
i) 
Let $N$ be a fixed positive integer that will be carefully chosen later.   Then   we have from (\ref{LU})
\begin{eqnarray}\label{t2}
\nonumber
\|\omega(t)\|_{B_{\infty,1}^0}
&\leq&
\sum_{j}\|\Delta_{j}\sum_{q}\tilde{\omega}_{q}(t)\|_{L^\infty}
\\\nonumber
&\leq&
\sum_{|j-q|\geq N}\|\Delta_{j}\tilde{\omega}_{q}(t)\|_{L^\infty}+\sum_{|j-q|<N}\|\Delta_{j}\tilde{\omega}_{q}(t)\|_{L^\infty}
\\
&:=&\hbox{I}+\hbox{II}.  
\end{eqnarray}

To estimate the first term we use Proposition \ref{Thm89}  and  the convolution inequality for the series
\begin{equation}
\label{t3}
\hbox{I}
\lesssim
2^{- N}\|\omega_{0}\|_{B_{\infty,1}^0}e^{CU(t)}.  
\end{equation}

To estimate the term $\hbox{II}$ we use two facts: the first
 one is that the operator $\Delta_{j}$ maps uniformly  $L^\infty$ into itself while the second  is  the $L^\infty$ 
 estimate (\ref{ben0}),
\begin{equation}\label{ben55}
\begin{aligned}
\hbox{II}
&\lesssim
\sum_{ |j-q|< N}\|\tilde{\omega}_{q}(t)\|_{L^\infty}
\\&
\lesssim e^{C_{0}t}\sum_{ |j-q|< N}\|\Delta_q\omega_0\|_{L^\infty}
\\&
 \lesssim  e^{C_{0}t}N\|\omega_0\|_{B_{\infty,1}^0}.  
\end{aligned}
\end{equation}
Combining this estimate with (\ref{ben55}),  (\ref{t3}) and (\ref{t2}) we obtain
$$
\begin{aligned}
\|\omega(t)\|_{B_{\infty,1}^0}
&\lesssim 2^{- N}e^{CU(t)}+  N e^{C_{0}t}
.  \end{aligned}
$$
Putting
$$
N=\Big[{CU(t)}\Big]+1,
$$
we obtain
$$
\|\omega(t)\|_{B_{\infty,1}^0}
\lesssim \big(U(t)+1\big) e^{C_{0}t}.  
$$
On the other hand we have
$$
\| u\|_{B_{\infty,1}^1}
\lesssim
\|u\|_{L^\infty}+\|\omega\|_{B^0_{\infty,1}},
$$
which yields in view of Proposition \ref{u/r}, 
$$
\begin{aligned}
\|u(t)\|_{B_{\infty,1}^1}&\lesssim\|u(t)\|_{L^\infty}+ \|\omega(t)\|_{B^0_{\infty,1}}
\\&
\leq C_{0}e^{\exp{C_{0}t}}+C_{0}e^{C_{0}t}\int^t_0\|u(\tau)\|_{B^1_{\infty,1}}d\tau.  
\end{aligned}
$$
Hence we obtain by Gronwall's inequality
$$
\|u(t)\|_{B_{\infty,1}^1}\leq C_{0}e^{\exp{C_{0}t}},
$$
which gives in turn 
$$
\|\omega(t)\|_{B_{\infty,1}^0}\leq C_{0}e^{\exp{C_{0}t}}.  
$$
This concludes the first part of Proposition \ref{ur1}.  

\

ii)
 Applying Proposition  \ref{Lems2} to the vorticity equation we get 
\begin{equation}\label{Duk0}
e^{-CU_{1}(t)}\|\omega(t)\|_{B_{p,1}^{\frac3p}}\lesssim \|\omega_{0}\|_{B_{p,1}^{\frac3p}}+\int_{0}^te^{-CU_{1}(\tau)}\|\omega\cdot\nabla u(\tau)\|_{B_{p,1}^{\frac{3}{p}}}d\tau.  
\end{equation}
As $\omega=\hbox{curl }u,$ we have
\begin{eqnarray}\label{Duk1}
\|\omega\cdot\nabla u\|_{B_{p,1}^{\frac{3}{p}}}&\lesssim &\|\omega\|_{B_{p,1}^{\frac3p}}\|\nabla u\|_{L^\infty}.  
\end{eqnarray}
Indeed, from Bony's decomposition we write
\begin{eqnarray*}
\|\omega\cdot\nabla u\|_{B_{p,1}^{\frac{3}{p}}}&\leq&\|T_{\nabla u}\cdot\omega\|_{B_{p,1}^{\frac3p}}+\|T_{\omega}\cdot\nabla u\|_{B_{p,1}^{\frac3p}}+\|\mathcal{R}(\omega,\nabla u)\|_{B_{p,1}^{\frac3p}}
\\
&\lesssim&\|\nabla u\|_{L^\infty}\|\omega\|_{B_{p,1}^{\frac{3}{p}}}+\|T_{\omega}\cdot\nabla u\|_{B_{p,1}^{\frac3p}}.  
\end{eqnarray*}
From the definition we write
\begin{eqnarray*}
\|T_{\omega}\cdot\nabla u\|_{B_{p,1}^{\frac3p}}&\lesssim&\sum_{q\in\NN}2^{q\frac{3}{p}}\|S_{q-1}\omega\|_{L^\infty}\|\nabla \Delta_{q}u\|_{L^p}\\
&\lesssim&\|\omega\|_{L^\infty}\sum_{q\in\NN}2^{q\frac{3}{p}}\|\Delta_{q}\omega\|_{L^p}\\
&\lesssim&\|\nabla u\|_{L^\infty}\|\omega\|_{B_{p,1}^{\frac{3}{p}}}.  
\end{eqnarray*}
We have used here the fact that for $p\in[1,\infty]$ and  $q\in\NN$ the composition operator $\Delta_{q}R: L^p\to L^p$ is continuous uniformly with respect to $p$ and $q,$ where $R$ denotes Riesz transform.   Combining (\ref{Duk0}) and (\ref{Duk1}) we find,
\begin{equation*}
e^{-CU_{1}(t)}\|\omega(t)\|_{B_{p,1}^{\frac3p}}\lesssim \|\omega_{0}\|_{B_{p,1}^{\frac3p}}+\int_{0}^te^{-CU_{1}(\tau)}\|\omega(\tau)\|_{B_{p,1}^{\frac{3}{p}}}
\|\nabla u(\tau)\|_{L^\infty}d\tau.  
\end{equation*}
 Gronwall's inequality yields
 $$
\|\omega(t)\|_{B^{{3\over p}}_{p,1}}
\leq
\|u_0\|_{B^{{3\over p}+1}_{p,1}}e^{C\int_0^t\|\nabla u(\tau)\|_{L^\infty}d\tau}\leq C_{0}e^{e^{\exp{C_{0}t}}}.  
$$
Let us  estimate the velocity.   We write 
\begin{eqnarray*}
\|u(t)\|_{B_{p,1}^{1+\frac3p}}&\lesssim&\|\Delta_{-1}u\|_{L^p}+\sum_{q\in\NN}2^{q\frac{3}{p}} 2^q\|\Delta_{q}u\|_{L^p}\\
&\lesssim& \|u(t)\|_{L^p}+\|\omega(t)\|_{B_{p,1}^{\frac3p}}.
\end{eqnarray*}
Thus it remains to estimate $\|u\|_{L^p}.  $ For $1<p<\infty,$ since Riesz transforms act continuously on $L^p,$ we get
\begin{eqnarray*}
\|u(t)\|_{L^p}&\leq& \|u_{0}\|_{L^p}+C\int_{0}^t\|u\cdot\nabla u(\tau)\|_{L^p}d\tau\\
&\lesssim&\|u_{0}\|_{L^p}+
\int_{0}^t\|u(\tau)\|_{L^p}\|\nabla u(\tau)\|_{L^\infty}d\tau.  
\end{eqnarray*}
It suffices now to use Gronwall's inequality.  

For the case $p=1,$ we write
$$
\begin{aligned}
\|u(t)\|_{L^1}
&\leq
\|\dot S_0u(t)\|_{L^1}
+
\sum_{q\geq 0}\|\dot\Delta_qu(t)\|_{L^1}
\\&
\lesssim
\|\dot S_0u(t)\|_{L^1}
+
\sum_{q\geq 0}2^{-q}\|\dot\Delta_q\nabla u(t)\|_{L^1}
\\&
\lesssim
\|\dot S_0u(t)\|_{L^1}
+
\|\omega(t)\|_{L^1}.  
\end{aligned}
$$
However, it is easy to see that $$
\|\omega(t)\|_{L^1}
\leq
\|\omega_0\|_{L^1}e^{\int_0^t\|\nabla u(\tau)\|_{L^\infty}d\tau}.  
$$
Concerning $\dot S_0u$ we use  the equation  on  $u$  leading to 
$$
\begin{aligned}
\|\dot S_0u(t)\|_{L^1}
&\lesssim
\|\dot S_0u_0\|_{L^1}
+\sum_{q\leq -1}\|\dot\Delta_q\mathbb{P}((u\cdot\nabla)u(t))\|_{L^1}
\\&
\lesssim
\|u_0\|_{L^1}
+\sum_{q\leq -1}2^q\|\dot\Delta_q(u\otimes u(t))\|_{L^1}
\\&
\lesssim
\|u_0\|_{L^1}+\|u(t)\|_{L^2}^2
\\&
\lesssim
\|u_0\|_{L^1}+\|u_0\|_{L^2}^2.  
\end{aligned}
$$
This yields 
$$
\|u(t)\|_{L^1}
\le C_{0}e^{e^{\exp{C_{0}t}}}.  
$$
The proof is now achieved.  

\end{proof}
\appendix
\section{Appendix}
\hskip 12pt
The following result describes  the anisotropic dilatation in Besov spaces.  
\begin{prop}\label{dilatation}
Let $f:\RR^3\to \RR$ be a function belonging to $B_{\infty,1}^0$ and denote by $f_{\lambda}(x_{1},x_{2},x_{3})=f(\lambda x_{1},x_{2},x_{3}).  $ Then, there exists an absolute constant $C>0$ such that for all $\lambda\in]0,1[$ 
$$
\|f_{\lambda}\|_{B_{\infty,1}^0}\leq C(1-\log\lambda)\|f\|_{B_{\infty,1}^0}.  
$$
\end{prop}
\begin{proof}
Let $q\geq-1,$ we denote by $f_{q,\lambda}=(\Delta_{q}f)_{\lambda}.  $ From the definition we have
\begin{eqnarray*}
\|f_{\lambda}\|_{B_{\infty,1}^0}&=&\|\Delta_{-1}f_{\lambda}\|_{L^\infty}+\sum_{j\in\NN}\|\Delta_{j}f_{\lambda}\|_{L^\infty}\\
&\leq&C\|f\|_{L^\infty}+\sum_{j\in\NN\atop q\geq-1}\|\Delta_{j}f_{q,\lambda}\|_{L^\infty}.  
\end{eqnarray*}
For $j,q\in\NN,$ the Fourier transform of $\Delta_{j}f_{q,\lambda}$ is supported in the set
$$
\Big\{|\xi_{1}|+|\xi'|\approx 2^j\quad\hbox{and}\quad \lambda^{-1} |\xi_{1}|+|\xi'|\approx 2^q \Big\},
$$
where $\xi'=(\xi_{2},\xi_{3}).  $
A direct consideration  shows that this set is empty \mbox{if $2^q\lesssim 2^j$ or $2^{j-q}\lesssim \lambda$}.   For $q=-1$ the set is empty if $j\geq n_{0},$ this last number is absolute.   Thus we get for an integer $n_{1}$
\begin{eqnarray*}
\|f_{\lambda}\|_{B_{\infty,1}^0}&\lesssim&\|f\|_{L^\infty}+\sum_{q-n_{1}+\log\lambda\leq j\atop
j\leq q+n_{1}}\|\Delta_{j}f_{q,\lambda}\|_{L^\infty}\\
&\lesssim&\|f\|_{L^\infty}+(n_{1}-\log\lambda)\sum_{q}\|f_{q,\lambda}\|_{L^\infty}\\
&\lesssim&\|f\|_{L^\infty}+(n_{1}-\log\lambda)\sum_{q}\|f_q\|_{L^\infty}\\
&\lesssim&(1-\log\lambda)\|f\|_{B_{\infty,1}^0}.  
\end{eqnarray*}
\end{proof}

The following  proposition describes the propagation of Besov regularity for transport equation.   
 \begin{prop}\label{Lems2}
Let $s\in]-1,1[, p,r\in [1,\infty]$ and u be a smooth divergence free vector field.    Let $f$ be a  smooth solution of the transport equation 
 $$
\partial_{t}f+u\cdot\nabla f=g,\, f_{|t=0}=f_0,
$$
such that $f_0\in B_{p,r}^s(\mathbb R^3)$ and 
$g\in{L^1_{\textnormal{loc}}}(\mathbb R_{+};B_{p,r}^{s}).  $     
\mbox{Then $\forall t\in\mathbb R_{+},$}
\begin{equation}\label{df}
\|f(t)\|_{B_{p,r}^s}     
\leq 
Ce^{CU_{1}(t)}
\Big(\|f_0\|_{B_{p,r}^s}+\int_{0}^te^{-CU_{1}(\tau)}\|g(\tau)\|_{B_{p,r}^{s}}d\tau\Big),
\end{equation}
where $ U_{1}(t)=\int_{0}^t\|\nabla u(\tau)\|_{L^\infty}d\tau$ and $C$ is a constant depending  on $s.  $

The above estimate  holds also true in the limiting cases: 
$$
s=-1,r=\infty, p\in[1,\infty])\quad\hbox{or} \quad s=1,r=1,p\in[1,\infty]
,$$
 provided that we change $U_{1}$ \mbox{by $U(t):=\|u\|_{L^1_{t}B_{\infty,1}^1.  }$}

In addition if $f=\textnormal{curl } u,$ then the above estimate $(\ref{df})$ holds true for \mbox{all $s\in[1,+\infty[.  $} 
\end{prop}
\begin{proof}
We will only restrict ourselves to the proof of the limiting \mbox{cases $s=\mp1.  $} The remainder cases are done for example in \cite{Ch1,vishik}.  

We start with localizing in frequency the equation leading to,
$$
\begin{aligned}
\partial_{t}\Delta_q f+(u\cdot\nabla)\Delta_qf
&=\Delta_qg+(u\cdot\nabla)\Delta_qf-\Delta_q\big(u\cdot\nabla f\big)
\\&
=\Delta_qg-[\Delta_q,u\cdot\nabla]f.  
\end{aligned}
 $$
Taking the $L^p$ norm, then the zero divergence of the flow gives  
$$
\|\Delta_qf(t)\|_{L^p}
\leq
\|\Delta_qf_0\|_{L^p}
+\int_0^t\|\Delta_qg\|_{L^p}d\tau
+\int_0^t\Big\|[\Delta_q,u\cdot\nabla]f\Big\|_{L^p}d\tau.  
$$
From Bony's decomposition, the commutator may be decomposed as follows 
\begin{eqnarray*}
[\Delta_q,u\cdot\nabla] f&=&\Delta_q \mathcal{R}(u^j,\partial_{j}f)+\Delta_qT_{\partial_j f}u^j
-T'_{\Delta_q\partial_j f}u^j+[\Delta_q, T_{u^j}]\partial_j f\\
&:=&\sum_{i=1}^4 \mathcal{R}^i_q,
\end{eqnarray*}
where $T'_{u}v$ stands for $T_{u}v+\mathcal{R}(u,v).  $To treat the first term
$\mathcal{R}^1_q,$ we write from the definition
$$
\mathcal{R}^1_q=\sum_{k\geq q-3}
\Delta_q\partial_{j}(\Delta_kf\widetilde\Delta_ku^j).  
$$
According to Bernstein inequalities we get for $s=-1,$
\begin{equation}\label{L1}
\sup_{q\geq-1}2^{-q}\Vert \mathcal{R}^1_q\Vert_{L^p}
\lesssim
\Vert f\Vert_{B^{-1}_{p,\infty}}\Vert u\Vert_{B^1_{\infty,1}}.  
\end{equation}
To estimate $\mathcal{R}^2_q,$ we write by definition
$$
\mathcal{R}^2_q=\Delta_qT_{\partial_j f}v^j=\sum_{\vert q-k\vert\leq 4}
\Delta_q(S_{k-1}\partial_j f\Delta_ku^j).  
$$
Applying Bernstein and Young inequalities leads to
\begin{eqnarray}\label{L2}
\nonumber\sup_{q}2^{-q}\Vert \mathcal{R}^2_q\Vert_{L^p}
\nonumber&\lesssim&
\sup_{q}2^{-q}
\|S_{q-1}f\|_{L^p}2^{q}\|\Delta_qu^j\|_{L^\infty}\\
\nonumber&
\lesssim&
\|u\|_{B_{\infty,\infty}^1}
\sup_{q}\sum_{-1\leq m\le q-2}2^{m-q}2^{-m}\|\Delta_mf\|_{L^p}
\\
&
\lesssim&\|f\|_{B_{p,\infty}^{-1}}\|u\|_{B_{\infty,\infty}^1}.  
\end{eqnarray}
It is easy to verify from the definition that $\mathcal{R}^3_q$ can be rewritten like
$$
\mathcal{R}^3_q=T'_{\Delta_q\partial_j f}u^j=\sum_{k\geq q-2}S_{k+2}\Delta_q\partial_j f\Delta_k u^j.  
$$
Thus applying Bernstein inequality one has
$$
2^{-q}\Vert \mathcal{R}^3_q\Vert_{L^\infty}
\lesssim
2^{-q}\Vert\Delta_q f\Vert_{L^p}
\sum_{k\geq q-2}2^{q-k}\, 2^k\Vert\Delta_k u\Vert_{L^\infty},
$$
Therefore we get from the convolution inequality
\begin{equation}\label{L3}
\sup_{q\geq-1}2^{-q}\Vert \mathcal{R}^3_q\Vert_{L^p}
\lesssim
\|f\|_{B_{p,\infty}^{-1}}\|u\|_{B_{\infty,\infty}^1}.  
\end{equation}
For the last term we write 
$$
\mathcal{R}^4_q=[\Delta_{q},T_{u^j}]\partial_j f=
\sum_{\vert k-q\vert\leq 4}[\Delta_{q},S_{k-1}u^j]\Delta_k\partial_j f.  
$$
The following is classical  (see for example \cite{Ch1}),
\begin{eqnarray*}
\Vert[S_{k-1}u^j,\Delta_q]\Delta_k\partial_j f\Vert_{L^\infty}
&\lesssim &2^{-q}\|\nabla S_{k-1}u\|_{L^\infty}\|\partial_{j}\Delta_{k}f\|_{L^p}\\
&\lesssim&2^{k-q}\Vert\nabla u\Vert_{L^\infty}\Vert\Delta_k f\Vert_{L^p}.  
\end{eqnarray*}
This yields
\begin{equation}\label{L4}
\sup_{q\geq-1}2^{-q}\Vert \mathcal{R}^4_q\Vert_{L^\infty}
\lesssim
\Vert f\Vert_{B^s_{p,\infty}}\Vert\nabla u\Vert_{L^\infty}.  
\end{equation}
Putting together the estimates (\ref{L1}), (\ref{L2}), (\ref{L3}) and (\ref{L4}) gives
$$
\sup_{q\geq-1}2^{-q}\Big\|[\Delta_q,u\cdot\nabla]f\Big\|_{L^p}
\lesssim
\|f\|_{B^{-1}_{p,\infty}}\|u\|_{B_{\infty,1}^1}.  
$$
This implies 
$$
\|f(t)\|_{B^{-1}_{p,\infty}}
\lesssim
\|f_0\|_{B^s_{\infty,\infty}}+\int_0^t\|g(\tau)\|_{B^{-1}_{p,\infty}}d\tau
+\int_0^t\|f(\tau)\|_{B^{-1}_{p,\infty}}\|u(\tau)\|_{B_{\infty,1}^1}d\tau.  
$$
It suffices now to use  Gronwall's inequality  in order  to the desired the result.  

Let us now move to the case $s=1$ that will be briefly explained.    We estimate $\mathcal{R}_{q}^1$ as follows
\begin{eqnarray*}
\sum_{q}2^q\|\mathcal{R}_{q}^1\|_{L^p}&\lesssim &\sum_{k\geq q-3}2^{q-k}2^k\|\Delta_{k}f\|_{L^p}2^k\|\widetilde\Delta_{k}u_{j}\|_{L^\infty}\\
&\lesssim&\|f\|_{B_{p,1}^1}\|u\|_{B_{\infty,\infty}^1}.  
\end{eqnarray*}
Concerning the second term we write
\begin{eqnarray*}
\sum_{q}2^q\|\mathcal{R}_{q}^2\|_{L^p}&\lesssim &\sum_{q}2^{q}\|S_{q-1}\partial_{j}f\|_{L^p}\|\Delta_{q}u^j\|_{L^\infty}\\
&\lesssim&\|\nabla f\|_{L^p}\|u\|_{B_{\infty,1}^1}\\
&\lesssim&\|f\|_{B_{p,1}^1}\|u\|_{B_{\infty,1}^1}.  
\end{eqnarray*}
The third and the last  terms are treated  similarly  to the first case.  
\end{proof}

\end{document}